\documentclass[11pt]{amsart}

\usepackage{amssymb,  amsfonts}
\usepackage{enumerate}
\usepackage[dvipsnames]{xcolor}
\usepackage{graphicx}
\usepackage{hyperref}
\usepackage{thmtools, thm-restate}
\usepackage{times}
\usepackage{yfonts}
\usepackage{egothic}
\usepackage[T2A,T1]{fontenc}

\renewcommand{\emph}[1]{\textbf{\textit{#1}}}

\usepackage{fancyhdr}
\newcommand{\bal}[1]{\begin{align*}#1\end{align*}}

\newcommand{\Z}{\mathbb{Z}} %ring of integers

\newcommand{\E}{\mathbb{E}}

\newcommand{\Cov}{\mathrm{Cov}}
\newcommand{\bP}{\mathbb{P}}

\newcommand{\Var}{\mathbb{V}\mathrm{ar}}

\newcommand{\eps}{\varepsilon}
\renewcommand{\phi}{\varphi}
\renewcommand{\P}{\bP} %another version of probability
\def\pto{\stackrel{p}{\longrightarrow}} %convergence in probability

\declaretheoremstyle[headfont=\normalfont\bfseries, 
postheadspace=\newline, bodyfont=\itshape, spaceabove=0.5cm,
spacebelow=0.5cm, notebraces=< >, shaded={bgcolor={blue!15!gray!5}}]{myLemmaStyle}

\declaretheoremstyle[headfont=\large\egothfamily, 
postheadspace=\newline, bodyfont=\itshape, spaceabove=0.5cm,
spacebelow=0.5cm, notebraces=| |, shaded={bgcolor={yellow!20}}]{myThmStyle}

\declaretheorem[heading=Theorem, parent = section, style= myThmStyle]{theo}

\declaretheorem[heading=Proposition, parent = section, style= myThmStyle]{propo}
\declaretheorem[heading=Lemma, sibling=theo, style= myLemmaStyle]{lemma}
\declaretheorem[heading=Corollary, sibling=theo, style= myLemmaStyle]{coro}

\theoremstyle{definition}
\newtheorem{defi}[theo]{Definition}

\begin{document}

\begin{center}
{\Large Limit theorems for statistics of non-crossing partitions}

\bigskip
Vladislav Kargin\footnote{{email:
vkargin@binghamton.edu; current address: 4400 Vestal Pkwy East, Department of Mathematics, Binghamton University, Binghamton, 13902-6000, USA}} 

\bigskip
Abstract: 
\end{center}

\begin{quote}
{\small
We study the distribution of several statistics of large non-crossing partitions. First, we prove the Gaussian limit theorem for the number of blocks of a given fixed size. In contrast to the properties of usual set partitions, we show that the number of blocks of different sizes are negatively correlated, even for large partitions. In addition, we show that the sizes of blocks in a given large non-crossing partition are distributed according to a geometric distribution and not Poisson, as in the case of usual set partitions. Next, we show that the size of the largest block concentrates at $\log_2 n$, and that after an appropriate rescaling, it can be described by the double exponential distribution. Finally, we show that the width of a large non-crossing partition converges to the Theta-distribution which arises in the theory of Brownian excursions.
}
\end{quote}

%TODO
%1. Write Abstract
%2. Add background material on the statistics of Dick Paths and trees. (!!!)
%3. Add reference to the work of Ardilla
%4. Possibly (!) add explanation of another bijection to NC pairings
%5. Add brief explanation why NC partitions are relevant, and their relation to random matrices, free probability, planar enumeration, and mathematical physics
%6. Add reference to a book by Stanley about Catalan numbers.
%7. Reference to Blanco and Petersen

\section{Introduction}

\subsection{Definition of NC partitions and bijections to other combinatorial structures}
Consider a partition of the ordered set $[n] = \{1, 2, \ldots, n\}$ into subsets (\emph{blocks}) $b_1, \ldots b_s$. This partition has a crossing if we can find elements $a < b < c < d$ such that $a$ and $c$ are in one block and $b$ and $d$ are in a different block. The partitions without crossings are called \emph{non-crossing (``NC'')} partitions. Their study was initiated in \cite{kreweras72} by Kreweras who described many of their properties. 

Later, non-crossing partitions and no-crossing pairings (NC partitions with blocks of size 2) have found many applications to problems in random matrix theory, in free probability, representation theory,  in theories of meanders and of Temperley-Lieb algebras (see, for example, \cite{biane98}, \cite{nica_speicher06}, \cite{simion2000}).

The class of non-crossing partitions belong to a broad family of Catalan discrete structures in the sense that the number of NC partitions of the set $[n]$ is given by the Catalan number $C_n$ and that NC partitions are connected by interesting bijections to other structures in this large family (see Stanley's book \cite{stanley2015} for a very complete description). 

\begin{figure}[htbp]
\centering
              \includegraphics[width=0.7\textwidth]{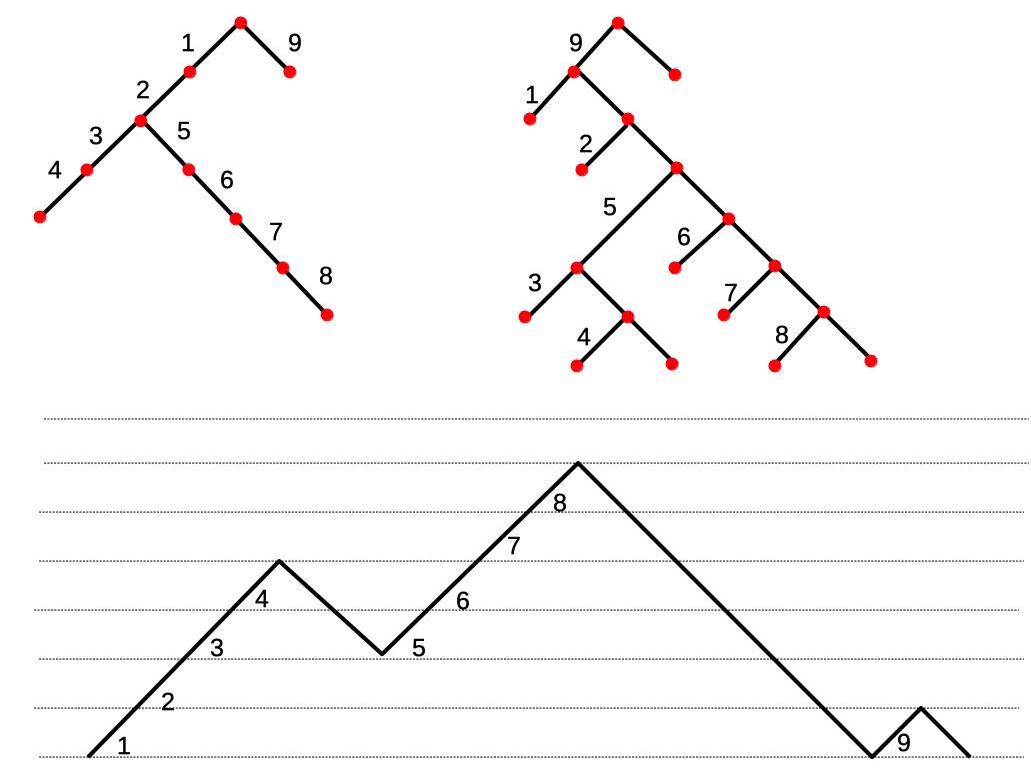}
              \caption{Discrete structures corresponding to a NC partition of $[9]$:
               $(1, 2, 5, 6, 7, 8), (3, 4), (9)$.}
              \label{figDyckPath}
\end{figure}

%We describe here several well-known bijections. Our goal in doing this is to see which parameters of these structures correspond to the number of blocks of a particular size. In this way we will see how our results about NC partitions can be translated to the corresponding results to related Catalan structures. 

Figure \ref{figDyckPath} illustrate some of these bijections. 
For example, NC partitions of $[n]$ are in bijection with the Dyck paths with $2n$ steps. These are paths on the $\Z^2$ lattice which start at $(0, 0)$, and then can go either by one step up or one step down, but are not allowed to go below the horizontal axis, and finally end up at $(2n, 0)$. Obviously, each Dyck path has $n$ ``up'' steps and $n$ ``down'' steps. Let the ``up'' steps be labeled as $1, \ldots, n$ in order of their appearance. Then, blocks of an NC partition correspond to un-interrupted stretches of  ``down'' steps, and elements of each block can be read off as the closest preceding ``up'' steps at the same level as the ``down'' steps in the stretch. 

The Dyck paths with $2n$ steps are in well-known bijection with rooted ordered trees with $n + 1$ vertex. This bijection can be described by a depth-first walk that explores the tree. The ``up'' and ``down'' steps of the Dyck path correspond to the steps of the walk that, respectively, increase or decrease the distance from the root. The superposition of these two bijections gives a bijection between NC partitions and rooted planar trees, and if we label every edge by the label of the corresponding ``up'' step in the Dyck path, then blocks of an NC partition corresponds to leaves of the tree. More precisely, each block corresponds to the set of the tree edges which are visited on the way back from a given leaf to either the node where the walker goes to a new, un-explored branch, or to the final node of the walk, the root. The size of the block corresponds to the length of this return trip. 

Another bijection is between Dyck paths with 2n steps and ordered rooted binary trees on $2n + 1$ vertices. Its details are described in Appendix.  Here we only note that in this bijection, the partition blocks correspond to the leaves at the end of right-directed edges. For a given leaf, we can construct a path that goes back over the right edges only. Then the partition block is given by labels of the left edges from the vertices in this path and the block length is the length of this path. See Figure \ref{figDyckPath} for illustration. 

It is worthwhile to note that while the blocks of NC partitions can be interpreted in terms of trees, they correspond to properties of the trees which have not been much investigated, in contrast to such popular properties as height and profile of trees.\footnote{The height of a tree is the maximal distance from the root to a leaf and the profile describes how many vertices of each given degree is contained in the tree.} 

\subsection{Statistics of usual set partitions}
Before exploring the properties of NC partitions, it is useful to recall results about the usual set partitions where the non-crossing condition is not imposed. These results can then be used as benchmark. 

We rely here on the book \cite{sachkov97} by V. N. Sachkov.

Let $\xi_n$ is the number of blocks in a random partition of $[n]$. (When we call an object  random, we mean that it is selected from a uniform distribution on the complete set of these objects.) Then, for large $n$, 
\bal{
\E\xi_n &= \frac{n}{\log n} (1 + o(1)), 
\\
\Var (\xi_n) &= \frac{n}{(\log n)^2} (1 + o(1)), 
}
and the distribution of the normalized random variable 
\bal{
\eta_n = \frac{\xi_n - \E(\xi_n)}{\sqrt{\Var(\xi_n)}}
}
converges to the standard normal distribution as $n \to \infty$  (Theorem 4.1.1 in \cite{sachkov97}).

Now, let the random variables $\kappa_n(l)$, $l = 1, \ldots, n,$ denote the number of blocks that have size $l$ in a random partition. Then, the distribution of $\kappa_n(l)$ has the expectation and the variance both equal to 
$\lambda_n = \frac{(r_n)^l}{l!}$, where  $r_n$ is the solution of the equation $r e^r = n$. 

%In particular, if we choose a random block in a random partition (uniformly among all blocks), then the distribution of its length is Poisson with parameter $r_n$.

If $l$ is fixed and $n$ is growing then the variances of random variables $\kappa_n(l)$,  $\lambda_n$ are also growing. We can define the normalized random variables 
\bal{
\hat\kappa_n(l) = \frac{\kappa_n(l) - \E \kappa_n(l)}{\sqrt{\Var(\kappa_n(l))}}.
}
For a fixed $s$-tuple $l_1 < \ldots < l_s$, the joint distribution of normalized random variables
$\hat\kappa_n(l_i)$ converges to the standard multivariate normal distribution (Theorem 4.2.1 in \cite{sachkov97}).

%If $l$ grows with $n$ then in a certain regime, we should have  $\lambda_n$ which stays close to the constant value, and it is natural to expect that the distribution of $\kappa_n(l)$ is asymptotically Poisson in this regime. This question is not addressed in \cite{sachkov97}. 

V. N. Sachkov discusses the distribution of the size of the maximum block, and shows that it is concentrated within a neighborhood of the point 
\bal{
e r_n - \log \sqrt{2 \pi e r_n} - \log (e - 1), 
}
and that in this domain it is close to the double exponential distribution (without any additional normalization). (For a more precise statement, see Theorem 4.5.2 in \cite{sachkov97}.) Note that $r_n$ is asymptotically close to $\log n$, hence in the first approximation, the size of the largest block is $e \log(n)$. 

\subsection{Statistics of Catalan structures} 
In this section we very briefly describe what is already known about statistics of non-crossing partitions and related Catalan structures. 

While direct studies about statistics of NC partitions  are not numerous and are essentially limited to a study by Simion in \cite{simion94}, there are many studies about statistics of Dyck paths and rooted ordered trees. Some results can be translated into the language of non-crossing partitions by using the bijection we described above.  The problem with this approach is that these results are often not very natural in the setting of NC partitions. 

 Now, here is a small list of the results. The study by Simion in \cite{simion94} investigated statistics of NC partitions arising from restricted growth functions.
In \cite{denise_simion95}, Denise and Simion derived generating functions for two statistics on Dyck paths, which they called the pyramid weight and the number of exterior pairs. The generating functions for many other statistics of Dyck paths were derived by Deutsch in \cite{deutsch99}. Blanco and Petersen in \cite{blanco_petersen2014} researched the joint distribution of the area under a Dyck path and the rank of this path.

%In this study, the rank of a path is defined as a number of blocks in a correspondent non-crossing partition (though the bijection between Dyck paths and NC partitions is different from what we described above). The area under the path can also be interpreted as a statistic of the partition. 

All these studies do not address the questions of asymptotic behavior of these statistics. In contrast, the asymptotic behavior was researched for statistics of rooted ordered trees, due to the importance of tree structures in the analysis of algorithms. In particular, for various families of trees, the researchers investigated their height, number of leaves, and the distribution of node degrees. 

In particular, in \cite{bkr72} and \cite{flajolet_odlyzko82}, it was shown that the tree height in many families has the expectation proportional to $c\sqrt{n}$, where $c$ is a constant specific to the family and $n$ is the size of the tree. The distribution of the height is a so-called theta distribution, named after theta distribution. It has connections to other areas of mathematics, see \cite{biane_pitman_yor01}.  

It appears that the statistics about Dyck paths do not correspond to such natural statistics of NC partitions as the total number of blocks and the number of blocks of a fixed size. However, the number of leaves of rooted ordered trees can be interpreted as the number of blocks in a NC partition, and we elaborate on the known results about the number of tree leaves by investigating its limit distribution and the distribution of the blocks with a fixed size. In addition, we will use a somewhat non-standard bijection to show that the height of trees has a relation to another characteristic of NC partitions which we call width. As a result we will establish that the width of a non-crossing partition has the theta-distribution.

\section{Limit theorems for random NC partitions}

\subsection{Number of parts}
Let us randomly select an NC partition $\pi$ from the uniform distribution on the set of all NC partitions of $[n]$ (denoted $NC(n)$). Then we can define several random variables associated with this random partition. Then we use $X_n$ to denote the number of blocks in $\pi$ and $Y_k^{(n)}$ to denote the number of blocks of length $k$ in $\pi$. Obviously, $X_n = Y_1^{(n)} + \ldots + Y_n^{(n)}$. It is natural to ask the question about the distribution of these random quantities for large $n$.

%%%%%%%%%%%%%%%%%%%%%%%%%%
%CLT for number of blocks
%
%%%%%%%%%%%%%%%%%%%%%%%%%%%%%
\begin{theo}
\label{theoBlocksCLT}
Let $X_n$ be the number of blocks in a random NC partition of $[n]$. Then \bal{
\E(X_n) &= \frac{n + 1}{2},
\\
\Var(X_n) &= \frac{n^2 - 1}{4(2n - 1)} \sim \frac{n}{8}.
}
 Let 
\bal{
Z_n = \frac{X_n - \E(X_n)}{\sqrt{\Var(X_n)}}.
}
Then, as $n \to \infty$, the cumulative distribution function of $Z_n$ converges to the standard Gaussian distribution function $\Phi(x)$. 
\end{theo}
Note that the expectation and variance are somewhat larger than corresponding quantities for usual partitions, which are $n/\log(n)$ and $n/\log^2(n)$, respectively.

\begin{proof} This result is easy because after a bijection it follows from analogous results for other Catalan structures. Namely, by the bijections above, the number of blocks in an NC partition corresponds to the number of leaves in a rooted planar tree and to the number of peaks in a Dyck path. The expectation and variance of these quantities are known. See, for example, section 6.1 in \cite{deutsch99}.

It is also known that the distribution of the number of leaves in rooted planar trees is asymptotically Gaussian. See Examples IX.24 and IX.25 on pp. 678 - 680 in \cite{flajolet_sedgewick2009}.
\end{proof}

\subsection{Distribution of the number of blocks of size $l$}

Consider a random NC partition of $[n]$. In the previous section, we have shown that on average, this partition has approximately $n/2$ blocks. How many of them have size $l \geq 1$?

First, let us define the relevant generating function. Let $N^{(l)}_{n, k}$ be the number of non-crossing partitions of $[n]$ that have $k$ blocks of size $l$. Then we define
\begin{equation}
\label{numberOfBlocksGF}
C^{(l)}(q, z) =\sum_{n\geq 0} \sum_{k = 0}^n N^{(l)}_{n, k} q^k z^n.
\end{equation}

\begin{figure}[htbp]
\centering
              \includegraphics[width=0.45\textwidth]{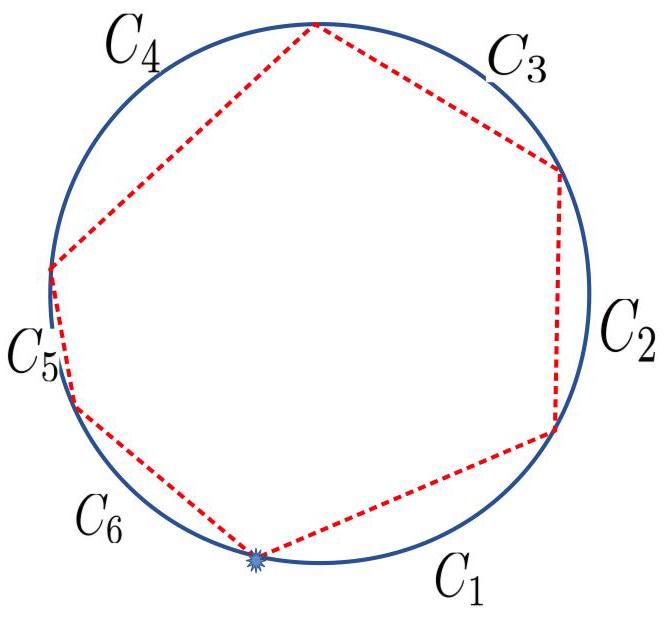}
              \caption{Construction of a non-crossing partition with a block of size 6}
              \label{figNCpartition}
\end{figure}

 \begin{theo}
 \label{theoGFforBlocks}
 The generating function $C^{(l)}(q, z)$ satisfies the following equation:
 \begin{equation}
 \label{blockGFEquation}
 C = \frac{1}{1 - zC} + (q - 1)(zC)^l. 
 \end{equation}
 \end{theo}
 \begin{proof}
 The equation is obtained by applying the symbolic transfer method by Flajolet and Sedgewick to the construction of an appropriate combinatorial class. The class $C$ here consists of all non-crossing partitions where the blocks of size $l$ are marked by marker $q$. The construction is given by the equation 
 \bal{
 C = &\epsilon + \mathrm{SET}_1(Z) \times \mathrm{SEQ}_1(C) + \ldots 
 + \mathrm{SET}_{l - 1}(Z) \times \mathrm{SEQ}_{l - 1}(C) 
 \\
 &+  q\mathrm{SET}_{l}(Z) \times \mathrm{SEQ}_{l}(C) 
 \\
 &+ \mathrm{SET}_{l + 1}(Z) \times \mathrm{SEQ}_{l + 1}(C) + \ldots 
 }
 Here $\epsilon$ denotes the empty partition, $Z$ is an atom (that is, an element of a partition), and $\mathrm{SET}_{k}(Z) \times \mathrm{SEQ }_{k}(C)$ corresponds to a block of size $k$ containing a marked element (``root''), together with a sequence of NC partitions which are nested between the  elements of this block. See Figure \ref{figNCpartition} for illustration. 
 
 Then by the symbolic method (Theorem I.1 and description of markers on p. 167 in \cite{flajolet_sedgewick2009}), this expression translates to the desired formula for the bivariate generating function:
 \bal{
 C^{(l)}(q, z) &= 1 + zC^{(l)}(q, z) + \ldots + \big(z C^{(l)}(q,z)\big)^{l - 1} +  q\big(z C^{(l)}(q,z)\big)^{l} 
 \\
 &+  \big(z C^{(l)}(q,z)\big)^{l + 1} + \ldots 
 = \frac{1}{1 - zC^{(l)}} + (q - 1)(zC^{(l)})^l. 
 }
 \end{proof}
 For $l = 1$, this equation can be solved explicitly, and we get an explicit formula for the generating function of the number of singletons.
 \begin{coro}
 The generating function for the number of singletons in non-crossing partitions is 
 \bal{
 C^{(1)}(q, z) = \frac{1 + (1 - q) z - \sqrt{1 - 2(1 + q) z + (-3 + 2q + q^2)z^2}}{2z\big(1 + (1 - q) z\big)}.
 }
 \end{coro}

%%%%%%%%%%%%%%%%%%%%%%%%%%%%%%%
%Theorem about Expectation, Variance and Covariance 
% of the number of blocks of a given size.
%
%%%%%%%%%%%%%%%%%%%%%%%%%%%%%%%

\begin{theo}
\label{theoMomentsOfBlocks}
Let $X_n^{(l)}$ denote the number of blocks of length $l\geq 1$ in a random NC partition of $[n]$. Then, 
\bal{
\E X_n^{(l)} = \frac{n}{2^{l + 1}} \prod_{j = 0}^l \Big(1 + \frac{2 - j}{2n - j}\Big).
}
For the variance, we have the asymptotic expression
\bal{
\Var X_n^{(l)} = \frac{n}{2^{2l + 3}}\Big[ 2^{l + 2} - (l - 1)^2 - 2\Big] + O(1),
}
where the $O$-term is for $n \to \infty$ and the constant implied in this term may depend on $l$.

The covariance of $X_n^{(k)}$ and $X_n^{(l)}$, for $k \ne l$, is given by 
\bal{
\Cov(X_n^{(k)}, X_n^{(l)}) =  -\frac{n}{2^{k + l + 3}}\Big[ 2 + (k - 1)(l - 1)\Big] + O(1).
}
\end{theo}

Note that here we have two differences with the similar result for usual set partitions. First, for large $n$ the distribution of partition blocks over sizes is not Poisson with mean $r_n \sim \log n$, but rather geometric with the expected number of blocks of size $l$ approximately $n/2^{l + 1}$. Second, the covariance between number of blocks of two different sizes is not negligible even for large $n$. In particular, after rescaling we cannot expect that random variables $X_n^{(k)}$ will form a Gaussian process with elements independent for different $k$. 

\begin{proof}[Proof of Theorem \ref{theoMomentsOfBlocks}]
Let us define $G(q, z) = zC^{(l)}(q, z)$ and $t = q - 1$. This function satisfies the equation
\bal{
\frac{G}{z} = \frac{1}{1 - G} + G^l t,
}
which we can re-write as 
\bal{
z = \frac{G}{\phi(G)},
}
where
\bal{
\phi(u) =\frac{1}{1 - u} + u^l t,
}
and $t$ is a parameter.  This expression is suitable for the Lagrange inversion formula for the coefficients in the series $G(z) = \sum_{n = 1}^{\infty} g_n z^n$, which gives 
\bal{
g_n = \frac{1}{n} [u^{n - 1}] \phi(u)^n 
= \frac{1}{n} [u^{n - 1}] \Big[ \frac{1}{1 - u} + u^l t\Big]^n,
}
where $[u^n]f(u)$ is notation for the coefficient before $u^n$ in the power series expansion of $f(u)$ around $u = 0$.
 
Since $C^{(l)} = G/z$, for coefficients in the series
 $C^{(l)}(q, z) = \sum_{n = 0}^{\infty} C^{(l)}_{n}(q) z^n$, we get
 \bal{
 C^{(l)}_{n}(q) =  \frac{1}{n + 1} [u^{n}] \Big[ \frac{1}{1 - u} + u^l (q - 1)\Big]^{n + 1}.
 } 
Then, by basic properties of generating functions, the expectation of the random variable  $X_n^{(l)}$ is 
\bal{
\frac{1}{c_n} \frac{d}{dq} C^{(l)}_{n}(q)\Big|_{q = 1}  
&= \frac{1}{c_n} [u^{n}] \Big[ \frac{1}{1 - u} + u^l (q - 1)\Big]^{n} u^l\Big|_{q = 1} 
\\
&=
\frac{1}{c_n} [u^{n - l}] (1 - u)^{-n}, 
}
where $c_n = C^{(l)}_{n}(1)$ is the $n$-th Catalan number. Hence, this expectation is 
\bal{
\bigg[\frac{1}{n + 1}\binom{2n}{n}\bigg]^{-1} \binom{2n - l - 1}{n - 1}
 = n \frac{n + 1}{2n} \frac{n}{2n - 1} \ldots \frac{n - l + 1}{2n - l},
 }
 which is in agreement with the expression in the statement of this theorem. 
 
 For the variance, we first compute the second factorial moment as 
 \bal{
\frac{1}{c_n} \frac{d^2}{dq^2} C^{(l)}_n(q)\Big|_{q = 1}  
&=
\frac{1}{c_n} n [u^{n - 2l}] (1 - u)^{-(n - 1)}
\\
&=
\frac{n(n - 1)}{2} \frac{n + 1}{2n - 1} \frac{n}{2n - 2} \ldots \frac{n - 2l + 1}{2n - 2l - 1}, 
}
 Then, we can get the following expansions for the expectation and the 2nd factorial moment:
 \bal{
 \E X_n^{(l)} &= \frac{n}{2^{l + 1}} \Big[ 1 + \frac{(l + 1)(4 - l)}{4} n^{-1} + O(n^{-2})\Big],
 \\
 \E [ X_n^{(l)} (X_n^{(l)} - 1)] &=  \frac{n^2 - n}{4^{l + 1}} \Big[ 1 + \frac{(2 l + 1)(3 - l)}{2} n^{-1} + O(n^{-2})\Big],
 }
which implies that 
\bal{
\Var(X_n^{(l)}) = \frac{n}{2^{2l + 3}} \Big[ 2^{l + 2} - 3 + 2l - l^2 + O(n^{-1})\Big],
}
and gives the second statement of the theorem.

Finally, in order to calculate the covariance of $X_n^{(k)}$ and $X_n^{(l)}$, we define a trivariate generating function, 
\begin{equation}
\label{triviariateGF}
C^{(k, l)}(p, q, z) =\sum_{n\geq 0} \sum_{a = 0}^n \sum_{b = 0}^n N^{(k, l)}_{n, a, b} \, p^a q^b z^n,
\end{equation} 
where $N^{(k, l)}_{n, a, b}$ is the number of NC partitions of $n$ that have $a$ blocks of size $k$ and $b$ blocks of size $l$.  

Then, by an argument similar to the argument is Theorem \ref{theoGFforBlocks}, we find that 
$C^{(k, l)}(p, q, z)$ satisfies the equation 
 \begin{equation}
 \label{equTriviariateGF}
 C = \frac{1}{1 - zC} + (p - 1)(zC)^k + (q - 1)(zC)^l. 
 \end{equation}
 Then, the coefficients in the expansion 
 \bal{
 C^{(k, l)}(p, q, z) = \sum_{n = 0}^{\infty} C^{(k, l)}_{n}(p, q) z^n 
 }
 can be calculated as 
  \bal{
 C^{(k, l)}_{n}(p, q) =  \frac{1}{n + 1} [u^{n}] \Big[ \frac{1}{1 - u} + u^k (p - 1) + u^l (q - 1)\Big]^{n + 1},
 } 
 and 
  \bal{
\E(X_n^{(k)} X_n^{(l)}) &=\frac{1}{c_n} \frac{\partial}{\partial p} \frac{\partial}{\partial q} C^{(k, l)}_n(p, q)\Big|_{p = 1, q = 1}  
\\
&=
\frac{1}{c_n} n [u^{n - (k + l)}] (1 - u)^{-(n - 1)}
\\
&=
\frac{n(n - 1)}{2} \frac{n + 1}{2n - 1} \frac{n}{2n - 2} \ldots \frac{n - (k + l) + 1}{2n - (k + l) - 1}. 
}
After some calculations, this leads to 
\bal{
\Cov(X_n^{(k)}, X_n^{(l)}) &= \E(X_n^{(k)} X_n^{(l)})  - \E(X_n^{(k)}) \E(X_n^{(l)})
\\
&=  - \frac{n}{2^{k + l + 3}} \Big(2 + (k - 1)(l - 1)\Big) + O(1),
}
which completes the proof of the third statement of the theorem. 
\end{proof}

%%%%%%%%%%%%%%%%%%%%%%%%%%
%CLT for number of blocks with fixed size
%
%%%%%%%%%%%%%%%%%%%%%%%%%%%%%%

Now we come to the question about the asymptotic distribution of the number of blocks of a given size. 

\begin{theo}
\label{blockSizeL_CLT}
Let $X_n^{(l)}$ denote the number of blocks of length $l\geq 1$ in a random NC partition of $[n]$. Define
\bal{
Z_n ^{(l)}= \frac{X_n^{(l)}- \E(X_n^{(l)})}{\sqrt{\Var(X_n^{(l)})}}.
}
Then, for every $l \geq 1$ as $n \to \infty$, the cumulative distribution function of $Z_n^{(l)}$ converges to the standard Gaussian distribution function $\Phi(x)$. 
\end{theo}

Before prooving this theorem, we summarize some tools from the book by Flajolet and Sedgewick. They are collected here for the convenience of the reader. 

 We say that that a function of complex argument $y(z)$ is an \emph{analytic generating function (analytic GF)} if it is analytic at zero and if  its expansion, 
 \begin{equation}
 \label{seriesY}
 y(z) = \sum_{n = 0}^{\infty} y_n z^n, 
 \end{equation}
 have real non-negative coefficients $y_0 = 0$, $y_n \geq 0$.
 \begin{defi}
 \label{defiStableSingularity}
 The analytic GF $y(z)$ is said to \emph{have a stable dominant singularity}\footnote{Flajolet and Sedgewick say that $y(z)$ belongs to the \emph{smooth implicit-function schema}.} at $z = r > 0$, if there exists a bivariate function $G(z, w)$ such that 
 \begin{equation}
 y(z) = G(z, y(z)), 
 \end{equation}
 and $G(z, y(z))$ satisfies the following conditions:
\begin{enumerate}[(A)]
 \item $G(z, w) = \sum_{m, n \geq 0} g_{m, n} z^m w^n$ is analytic in a domain $|z| < R$ and $|w| < S$ for some $R, S > 0$.
 \item Coefficients $g_{m, n}$ are non-negative reals, $g_{0, 0} = 0$, $g_{0, 1}\ne 1$ and $g_{m, n} > 0$ for some $m$ and for some $n \geq 2$. 
 \item  The number $r < R$ and there exists $s$ such that $0 < s < S$ such that 
 \begin{align}
 G(r, s) &= s,
 \\
 G_w(r, s) &= 1,
 \end{align} 
\end{enumerate}
 \end{defi}
 We say that $G(z, w)$ is the \emph{characteristic function} of $y(z)$.
 
 The condition in (C) is aimed to ensure that $r$ is a singularity of $y(z)$ with $y(r) = s$. Then the conditions in (A) and (B), especially the non-negativity of the coefficients, ensure that this singularity is a quadratic singularity with the smallest absolute value among all singularities of $y(z)$ (which is why we call it ``the stable dominant singularity''). This statement is explicated in the following theorem. For the case of polynomial or entire $G(z,w)$ with non-negative coefficients the fact that $r$ is a dominant singularity can be found in the classic book by Hille (\cite{hille62}, Theorem 9.4.6 on p. 274 of volume I) without mention that the singularity is quadratic. In a more general form it was formulated first in \cite{bender74} with an error in the set of conditions (see counterexample in \cite{canfield84}) and proved in correct form in \cite{meir_moon89}.
 
 %%%%%%%%%%%%%%%%%%%%%%%%%
 %Flajolet-Sedgewick theorem about expansion of functions with 
 % a stable singularity
 %
 %%%%%%%%%%%%%%%%%%%%%%%%%%%%
 
 \begin{theo}
 \label{theoStableSingularity}
 Let $y(z)$ be an analytic GF that has a stable singularity at $r$ with the characteristic function $G(z, w)$. Then the series in (\ref{seriesY}) converges at $z = r$ and 
 \begin{equation}
 y(z) = s - \gamma \sqrt{1 - z/r} + O(1 - z/r),
 \end{equation}
 in a neighborhood of $z = r$, where $s = y(r)$ and 
 \begin{equation}
 \gamma = \sqrt{\frac{2 r G_z(r, s)}{G_{ww}(r, s)}}.
 \end{equation}
 \end{theo}

This theorem allows to extract information about the coefficients in the expansion of $y(z)$. One additional condition is needed. An analytic generating function $y(z)$ is called \emph{aperiodic} if for some $i < j < k$, the coefficients $y_i, y_j, y_k$ are all non-zero and $\textrm{gcd}(j - i, k - i) = 1.$

\begin{coro}
If analytic GF $y(z)$ satisfies the conditions of the previous theorem and aperiodic then 
\begin{equation}
[z^n] y(z) = \frac{\gamma}{2\pi n^3} r^{-n} \Big(1 + O(n^{-1}\Big).
\end{equation}
\end{coro} 
 
Now, let us consider the bivariate generating function $H(z, q)$ and let us consider the probability distribution with the following probability generating function:

\begin{equation}
p_n(q) = \frac{[z^n] H(z, q)}{[z^n] H(z, 1)}.
\end{equation}

We are interested in sufficient conditions on the generating function that ensure that this probability distribution converges (after normalization) to the standard Gaussian law. 
These conditions are given by Proposition IX.17 in Flajolet-Sedgewick, which we repeat below. 

Recall that the variability operator $\mathbb{V}$ is defined as 
\bal{
\mathbb{V}[B(q)] = B''(1) + B'(1) - \big( B'(1)\big)^2,
}
provided that $B(1) = 1$.

Let $H(z, q)$ be a bivariate generating function, analytic at $(0, 0)$ and suppose that it solves the equation $y = \Phi(z, y; q)$, where $\Phi(z, y; q)$ is a polynomial of degree at least $2$ in $y$. Let us define the following conditions that we can impose on $H(z, q)$.

\begin{enumerate}[(I)]
\item The function $y(z) = H(z, 1)$ has a stable dominant singularity at $z = \rho$ with the characteristic function $G(z, w) = \Phi(z, w; 1)$. 
\item There is a function $z = \rho(q)$ (``singularity movement function'') that solves the equation obtained from polynomial equations  
 \begin{align}
 \Phi(z, y; q) &= y,
 \\
 \Phi_y(z, y; q) &= 1,
 \end{align} 
 by elimination of variable $y$. This function is analytic in a neighborhood of $q = 1$ and $\rho(1) = \rho$, where $\rho$ is as in condition I above.
 \item The function $\frac{\rho}{\rho(q)}$ satisfies the variability condition: 
 \begin{equation}
 \mathbb{V}\Big[\frac{\rho}{\rho(q)}\Big] > 0.
 \end{equation}
\end{enumerate}

\begin{propo}
\label{GFtoGauss}
Let $H(z, q)$ be a bivariate generating function, analytic at $(0, 0)$ and suppose that it solves the equation $y = \Phi(z, y; q)$, where $\Phi$ is a polynomial of degree at least $2$ in $y$. \\
Assume that Conditions I, II, III above are satisfied. 

Then the probability distribution with the probability generating function \begin{equation}
p_n(q) = \frac{[z^n] H(z, q)}{[z^n] H(z, 1)}.
\end{equation}
has an asymptotic Gaussian distribution. 
\end{propo}

Now, we are able to proceed to the proof of our result about the asymptotic Gaussian distribution of the random variables $X_n^{(l)}$.

\begin{proof}[Proof of Theorem\ref{blockSizeL_CLT}]
In our case the bivariate generating function $C_l(z, q)$ is defined in equation
(\ref{numberOfBlocksGF}), and it satisfies the equation (\ref{blockGFEquation}). It is convenient to define $H(z, q) = zC_l(z, q)$, which satisfies the equation 
\bal{
H = H^2 + z + (H^l + H^{l + 1})(q - 1) z.
}
Thus, in terms of Proposition \ref{GFtoGauss}, we can use 
\bal{
\Phi(z, y; q) = y^2 + z + (y^l + y^{l + 1})(q - 1) z.
}
For $q = 1$ the solution has a stable dominant singularity at $\rho = 1/4$. (Condition I is satisfied.)

If $q \ne 1$, then the equation $\Phi_y(z, y; q) = 1$ leads to 
\begin{equation}
\label{y2z}
z = \frac{1 - 2 y}{(q - 1) [l + (l + 1) y]y^{l - 1}}.
\end{equation}
After substituting this expression into equation $\Phi(z, y; q) = y$ and simplifying, we are led to the following equation for $s = y(q)$ (the value of $H(z, q)$ at the branching point).
\bal{
-1 + 2s + (q - 1) s^l \Big[(l- 1) (1 - s^2) + 2s\Big] = 0.
}
We are interested in the expansion 
\bal{
s(q) = \frac{1}{2} + a(q - 1) + b(q - 1)^2 + O\Big((q - 1)^3\Big),
}
and we calculate
\bal{
a &= - \frac{1}{2^{l + 3}}(3 l + 1),
\\
b &= \frac{1}{2^{2l + 5}}(9 l^3 + 17 l + 6).
}
Then we get the expansion for $\rho(q)$ by using (\ref{y2z}),
\bal{
\rho(q) &= \frac{1}{4} - \frac{3}{2^{3 + l}} (q - 1) - \frac{l}{2^{7 + 2l}}( - 21 + 
77l - 27 l^2 + 27 l^3) (q - 1)^2 
\\
&+   O\Big((q - 1)^3\Big).
}
In other words, we found the function $\rho(q)$ which is analytic in the neighborhood of $q = 1$ and satisfies Condition II of the proposition. 
It is a routine calculation to check that Condition III is also satisfied. Hence, the normalized coefficients in $[z^n] H(z, q)$ have an asymptotic Gaussian limit. 
\end{proof}

%%%%%%%%%%%%%%%%%%%%%%%%%%%
%The size of the largest block
%
%%%%%%%%%%%%%%%%%%%%%%%%%%

\subsection{The size of the largest block}

First, we show that the largest block in a typical NC partition has size $\log_2 n$.

\begin{theo}
Let $L_n$ denote the size of the largest block in a random NC partition of $[n]$. 
Then, as $n \to \infty$, 
\bal{
\frac{L_n}{\log_2 n} \pto 1,
}
where convergence is in probability. 
\end{theo}
Note that $\log_2 n = (\log_2)^{-1} \log n \approx 1.443 \log n$, and therefore the largest block in a NC partition is on average shorter than in a usual partition where it is around $e \log n \approx 2.718 \log n$
\begin{proof}
\bal{
\Pr\Big[L_n \geq (1+ \eps) \log_2 n\Big]
&=\Pr\Big[\bigcup_{l\geq l(n)}^{\infty} \{X_n^{(l)} \geq 1\}\Big]
\\
&\leq \sum_{l\geq l(n)}^{\infty} \Pr\Big[\{X_n^{(l)} \geq 1\}\Big],
}
where $l(n) = (1 + \eps) \log_2 n$.
Note that by the Markov inequality, 
\bal{
\Pr[X_n^{(l)} \geq 1] \leq \E X_n^{(l)} = \frac{n}{2^{l + 1}} \prod_{j = 0}^l \Big(1 + \frac{2 - j}{2n - j}\Big) < \frac{n}{2^l}, 
}
for $n \geq 2$. Then, the inequality above becomes
\bal{
\Pr\Big[L_n \geq (1 + \eps) \log n\Big] < 2n^{- \eps}
\to 0 
}
as $n \to \infty$.

In the opposite direction, we can write, 
\bal{
\Pr\Big[L_n < (1 - \eps) \log_2 n\Big]
&=\Pr\Big[\bigcap_{l\geq l(n)}^{\infty} \{X_n^{(l)} = 0 \}\Big]
\leq \Pr\Big[X_n^{(l(n))} = 0\Big],
}
where $l(n) = \lceil (1 - \eps) \log_2 n\rceil$.
We can estimate the probability on the right hand side of the inequality by the Chebyshev inequality,
\bal{
\Pr\Big[X_n^{(l(n))} = 0\Big] \leq \Pr\Big[|X_n^{(l(n))} - \E X_n^{(l(n))}| \geq  \E X_n^{(l(n))} - 1 \Big] \leq \frac{\Var(X_n^{(l(n))})}{(\E X_n^{(l(n))} - 1)^2}. 
}
From the asymptotic formulas for the expectation and the variance we obtain, 
\bal{
\Var(X_n^{(l(n))}) &\sim \E(X_n^{(l(n))}) \sim n^{\eps}, 
}
which implies that 
\bal{
\Pr\Big[X_n^{(l(n))} = 0\Big] \leq O(n^{-\eps}) \to 0,
}
as $n \to \infty$. This concludes the proof. 
\end{proof}

The next step is to determine the distribution of the largest block size as it deviates from $\log_2 n$. As it turns out, for large $n$, the largest block size distribution depends on how $n$ places itself with respect to powers of $2$. We use notation $\lfloor t \rfloor$  for the largest integer $\leq t$, and $\{t\} = t - \lfloor t \rfloor$ for the fractional part of $t$.

\begin{theo}
Let $L_n$ denote the size of the largest block in a random NC partition of $[n]$. Let an integer $k = \lfloor \log_2 n \rfloor + x$, and define $\alpha(n) = 2^{\{\log_2 n\}}\in [1, 2)$. Then, as $n \to \infty$, 
\bal{
\P[L_n \leq k] =  \exp\Big( -\alpha(n) 2^{-(x + 1)}\Big) \Big( 1 + O\big(n^{-1}\log^2 n\big)\Big).
}
\end{theo}

\begin{proof}
First, we find the generating function for NC partitions with blocks whose length is $\leq k$. The symbolic formula for the class of these partitions is 
\bal{
C^{(k)} = \eps + SET_1(Z) \times SEQ_1(C^{(k)}) + \ldots + SET_k(Z) \times SEQ_k(C^{(k)}).
}
This leads to the following equation for the generating function:
\bal{
C^{(k)}(z) &= 1 + z C^{(k)}(z) + \ldots + \Big[zC^{(k)}(z)\Big]^k
\\
&=\frac{1 - \Big[zC^{(k)}(z)\Big]^{k + 1}}{1 - zC^{(k)}(z)},
}
or 
\bal{
\Big[zC^{(k)}(z)\Big]^{k + 1} - z C^{(k)}(z)^2 + C^{(k)}(z) - 1 = 0.
}
Let us use the notation 
\bal{
y(z) = zC^{(k)}(z).
}
Then the equation for $y$ is $y = P(z, y)$, where $P(z, y) = z + y^2 - z y^{k + 1}$. 

This leads us to the situation described in Definition \ref{defiStableSingularity} and Theorem \ref{theoStableSingularity}, where $P(z, y)$ is the characteristic function for $y(z)$. 
 Theorem \ref{theoStableSingularity} is useful for us because it will allow us to determine the expansion of the generating function near the dominant singularity, and this expansion and the transfer theorems of the symbolic method will give us  the asymptotic expression for the number of NC partitions of $[n]$ with all block sizes $\leq k$. 

The singularity $(z_0, y_0 = y(z_0)$  solves the characteristic system:
\bal{
y &= z + y^2 - z y^{k + 1},
\\
1 &= 2y - (k + 1)z y^k.
}
The first equation of the system gives 
\bal{
z = \frac{y(1 - y)}{1 - y^{k + 1}},
}
and after plugging this expression into the second equation we obtain:
\bal{
y = \frac{1}{2}\Big[1 + k y^{k + 1} - (k - 1) y^{k + 2}\Big].
}
The solution for this equation is 
\bal{
y_0 = \frac{1}{2} + \frac{k + 1}{2^{k + 3}} + O\Big(\frac{k}{2^{2k}}\Big).
}
And then, 
\bal{
z_0 = \frac{y_0(1 - y_0)}{1 - y_0^{k + 1}} = \frac{1}{4}\Big(1 + \frac{1}{2^{k + 1}}\Big) + O\Big(\frac{k^2}{2^{2k}}\Big),
}
In order to apply Theorem \ref{theoStableSingularity} we also compute
\bal{
\gamma = \sqrt{\frac{2 z_0 P_z(z_0, y_0)}{P_{yy}(z_0, y_0}} = \frac{1}{2} + O\Big(\frac{k^2}{2^{k}}\Big)
}
and conclude that 
\bal{
y(z) = y_0 - \gamma \sqrt{1 - z/z_0} + O(1 - z/z_0).
}
Using the power expansion for the square root and Theorem VI.4 in \cite{flajolet_sedgewick2009} to justify that the error term in the formula for $y(z)$ can be neglected, we find that 
\bal{
[z^n] y(z) = \gamma \Big[\frac{1}{2\sqrt{\pi}n^{3/2}} + O(n^{-5/2})\Big]\Big(\frac{1}{z_0}\Big)^n.
}
For Catalan numbers (total number of NC partitions of $[n]$) we have the asymptotic approximation
\bal{
C_n = 4^n\Big[\frac{1}{\sqrt{\pi n^{3/2}}} + O\Big(\frac{1}{n^{5/2}}\Big)\Big].
}
Hence we have the following estimate:
\bal{
\P\{L_n \leq k\} &= \frac{[z^n]C^{(k)}(z)}{C_n} = \frac{[z^{n + 1}]y(z)}{C_n} 
\\
&= 4^{-(n + 1)}z_0^{-(n+1)}\Big(1 + O(n^{-1}) + O(k^2/2^k)\Big).
\\
&= \Big[\Big(1 + \frac{1}{2^{k + 1}}\Big) + O\Big(\frac{k^2}{2^{2k}}\Big)\Big]^{-(n+1)}\Big(1 + O(n^{-1}) + O(k^2/2^k)\Big).
}
We have assumed that $k = \lfloor \log_2 n \rfloor + x$ and defined $\alpha(n) = 2^{\{\log_2 n\}}$,  hence $2^{k + 1} = 2^{\lfloor \log_2 n \rfloor + x + 1} =2^{x + 1} \frac{n}{\alpha(n)}$. Plugging this into the previous expression, we find that 
\bal{
\P\{L_n \leq k\} &=  \Big[1 + \alpha(n) 2^{-(x + 1)} \frac{1}{n} + O\Big(\frac{\log^2 n}{n^2}\Big)\Big]^{-(n+1)}\Big(1 + O\big(\frac{\log^2 n}{n}\big)\Big)
\\
&=\exp\Big( - \alpha(n) 2^{-(x + 1)}\Big) \Big( 1 + O\big(n^{-1}\log^2 n\big)\Big).
}
\end{proof}
%%%%%%%%%%%%%%%%%%%%%%%
%The part below is an interesting question but it have not yielded to 
%analysis yet.
%
%%%%%%%%%%%%%%%%%%%%%%%%%
%\subsection{A typical part}
%Let $s$ be a randomly chosen point, $0 \leq s \leq n - 1$ and $B$ be a block that contains this point. What is the  distribution of the size of this random block of the random NC partition?
%
%This seems to be a difficult question and we address a related question about the \emph{range} of the block. 
%
%If $B = (s_1, \ldots , s_k)$, and $0 \leq s_1 < s_2 < \ldots < s_k \leq n - 1$, then we say that the \emph{size} of the block is $k$, 
%and the \emph{range} of the block is $ \min(s_k - s_1, s_1 + n - s_2, \ldots s_{k - 1} + n - s_k). $ 
%
%Intuitively, if we think about the points as placed on the unit circle, then the range of the block is the length of the smallest arc of the circle that covers the block. 
%
%\begin{figure}[htbp]
%\begin{minipage}[b]{0.4\linewidth}
%\centering
%              \includegraphics[width=\textwidth]{bijectionPartition.png}
%              \caption{}
%              \label{figBijectionPartition}
%\end{minipage}
%\hspace{0.5cm}
%\begin{minipage}[b]{0.6\linewidth}
%\centering
%              \includegraphics[width=\textwidth]{bijectionPairing.png}
%              \caption{}
%              \label{figBijectionPairing}
%\end{minipage}
%\end{figure}
%
%
%It will be useful here to use the bijection between NC partitions of $[n]$ and NC pairings of $[2n]$ given by the so-called doubling construction. This construction is illustrated in Figures \ref{figBijectionPartition} and \ref{figBijectionPairing}.

\section{Width of non-crossing partitions}
Let us think about the elements of the set $[n] = \{1, \ldots, n\}$ as points on the line of real numbers.  If $x_1 < x_2 < \ldots  x_k$ are points in a block $b$, then we can represent the block by semicircles $S_{[x_1, x_2]}$, $S_{[x_2, x_3]}$, \ldots,  $S_{[x_{k - 1}, x_k]}$, where $S_{[a,b]}$ denotes a semicircle in the upper half-plane with the diameter $[a, b]$. If the block has only one element $x$, then it is represented by a small vertical interval of length 1/2 that sticks out into upper half-plane at abscissa $x$.

Note that if we draw semicircles and intervals for all blocks of an NC partition $\pi$, they will be non-intersecting (except trivially for the same block at the real line). We say that this system of semicircles and intervals represents the partition $\pi$. 

Then for every half-integer point $x + \frac{1}{2}$, we can calculate \emph{width of the partition $\pi$ at $x$} as the number of intersections of the vertical line with abscissa $x +\frac{1}{2}$ and the semicircles in the graphical representation of $\pi$. (This vertical line never intersects the vertical intervals of singleton blocks.) Let us denote it $w_x(\pi)$. Then, the \emph{width} of an NC partition $\pi$ of the set $[n]$ is the maximum of $w_x(\pi)$ over all possible $x$, 
\bal{
w(\pi) = \max\{w_x(\pi) | x = 1, \ldots, n - 1\}. 
}

For the asymptotic distribution of the width we have the following interesting theorem. 
\begin{theo}
\label{theoWidth}
Let $W_n$ be the width of a random NC partition of the set $[n]$, and let $0 < \alpha < x <\beta < \infty$. Then, uniformly in $x$, 
\bal{
\P(W_n \geq  x \frac{\sqrt{n}}{2} ) = \Theta(x) + O(n^{-1}),
}
where 
\bal{
\Theta(x) = \sum_{j = 1}^{\infty} e^{-j^2 x^2}(4 j^2 x^2 - 2). 
}
\end{theo}

In particular, this implies that the expected width is $\E W_n = \frac{1}{2} \sqrt{\pi n} - \frac{3}{4} + o(1)$. More generally, the $r$-th moment of the distribution is given by the expression: 
\bal{
\E W_n^r = r (r - 1) \Gamma (r/2) \zeta(r)\Big(\frac{\sqrt{n}}{2}\Big)^r,
}
where $\zeta(z)$ is the Riemann zeta-function. For these formulas, see Example V.8 and Proposition V.4 on pp. 326 - 329 in \cite{flajolet_sedgewick2009}.

In the case when non-crossing condition is not imposed, the width for pairings of $2n$ points on the line was analyzed.  In this case, it was found that width converges to $n/2$ (see \cite{loz85}). Thus, the typical width of pairings without non-crossing condition is significantly larger. In addition, in the case of general pairings there is some research on the width $w_x$ as a random process in $x$ (see Example V.10 on p. 333 in \cite{flajolet_sedgewick2009} and references within). This is open in the non-crossing case.  

\begin{figure}[htbp]
\centering
              \includegraphics[width=0.7\textwidth]{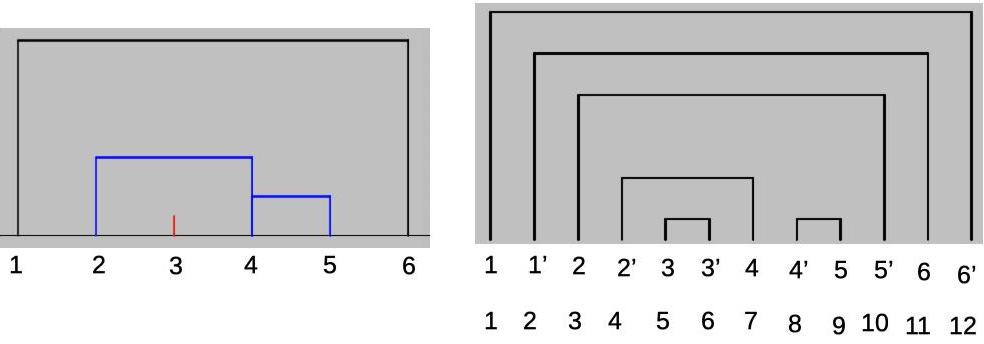}
              \caption{An example of doubling construction}
              \label{figDoublingConstruction}
\end{figure}

We start the proof of Theorem \ref{theoWidth} by noting that a pairing $P$ on $[2n]$ is a particular case of a partition and, therefore, its width $w(P)$ is well defined. In  addition, there is a bijection from NC partitions of $[n]$ to NC pairings of $[2n]$  by means of a so-called doubling construction. It is defined in \cite{nica2016} and it is equivalent to a bijection between Dyck paths  and NC partitions described in \cite{stump2013}.\footnote{This path-partition bijection is different from the bijection that we described in the beginning of this paper.}

An illustration of the doubling construction in Figure \ref{figDoublingConstruction} is hoped to be sufficient for understanding of how it works. In this construction, every point is doubled, so instead of a point $k$, we have two points $k$ and $k'$. If a block is not a singleton, then each line of a block in an NC partition $\pi$ corresponds to two lines in the corresponding NC pairing $P$. A vertical line interval of a singleton block $\{k\}$ in $\pi$  corresponds to a single line between $k$ and $k'$ in the NC pairing $P$. For formal definition, see formula (2.4) in \cite{nica2016}.

 We have the following lemma. 

\begin{lemma} 
If NC partition $\pi$ and NC pairing $P$ are related by the doubling construction then $w(\pi) = \lfloor w(P)/2 \rfloor$.
\end{lemma}
\begin{proof}
Let $\pi$ be a NC partition of $[n]$.  Consider $1 \leq x \leq n$ and let $w_x(\pi) = k$, that is, the vertical line with abscissa $x + \frac{1}{2}$ intersects $k$ partition lines $L$. Then, $w_{2x}(P) = 2w_{x}(P) = 2k$,  since the vertical line with abscissa $2x + \frac{1}{2}$ intersects exactly those $2k$ lines that correspond to the lines $L$ under the doubling construction. 

This implies, in particular, that $w(P) \geq 2 w(\pi)$. Suppose that the maximum of $w_y(P)$ is reached at some $y_0$. If $y_0 = 2x$ then $w(P) = w_{2x}(P) = 2 w_x(\pi) \leq 2\max_x w_x(\pi) = 2 w(\pi)$, hence $w(P) = 2 w(\pi)$. Alternatively, if $y_0 = 2x - 1$, then $w_{2x - 1}(P)$ must be odd by properties of pairings and $2 w_x(\pi) = w_{2x}(P) \geq w_{2x - 1}(P) - 1$ since  no more than one crossing with a vertical line can be eliminated when the line's abscissa changes from $2x - \frac{1}{2}$ to $2x + \frac{1}{2}$. Hence $w(\pi) \geq w_x(\pi) \geq \lfloor w_{2x - 1}(P)/2 \rfloor = \lfloor w(P)/2 \rfloor$, which together with the inequality above implies that  $\lfloor w(P)/2 \rfloor = w(\pi)$.

%Can $w(P)$ be significantly larger than $2w(\pi)$? For this question, we need to check the widths $w_{2x - 1}(P)$.

%In order to calculate $w_{2x - 1}(P)$, consider two cases. First, suppose $x$ is a element of a block with at least two elements. Then $w_{2x - 1}(P) = 2 w_x(\pi) - 1$, since all of the $k$ lines that were intersected in the partition case by line $x + \frac{1}{2}$ will be doubled by the doubling construction, and all of the resulting lines, except one, will be intersected by the line $2x - \frac{1}{2}$. 

%Note, however, that in this case $w_{2x - 1}(P)$ can never equal the height of the pairing $P$, since the height of $P$ at either $2x - 2$ or $2x$ will be larger. Indeed, since the partition block with $x$ has more than 2 elements, shifting the vertical line either to the left or to the right will increase the number of crossed lines by 1.  

%The second case occurs when $x$ is a singleton in the partition $\pi$. In this case $w_{2x - 1}(P) = 2w_x(\pi) + 1 = 2k + 1$ since one additional line is crossed by the vertical line with abscissa $2x - \frac{1}{2}$, and this line is not result of doubling of lines in the partition $\pi$. 

%This implies that either $w(P) = 2 w(\pi)$,  or $w(P) = 2 w(\pi) + 1$.

\end{proof}

\begin{proof}[Proof of Theorem \ref{theoWidth}]
By preceding arguments, it is enough to prove a corresponding result for the width of non-crossing pairings. 

For NC pairings, the number of intersections of the vertical line $L_{x + \frac{1}{2}}$ and the semicircles of a pairing $P$ equals the number of pairs $(a, b) \in P$ such that $a < x + 1/2 < b$, that is, the number of pairing arcs that have already started but have not yet been closed. Note that in the standard bijection of NC pairings and Dyck paths, the start of a pair correspond to a step up and the end of a pair corresponds to a step down. Hence, the width of a pairing at $x$ equals to the height of the corresponding Dyck path at $x$. 

Then the conclusion of the theorem follows from the known results about the height distribution of Dyck paths and rooted planar trees. (See \cite{fgor93} and Example V.8 on p. 326 - 330 in \cite{flajolet_sedgewick2009}).

\end{proof}

\begin{figure}[htbp]
\centering
              \includegraphics[width=0.7\textwidth]{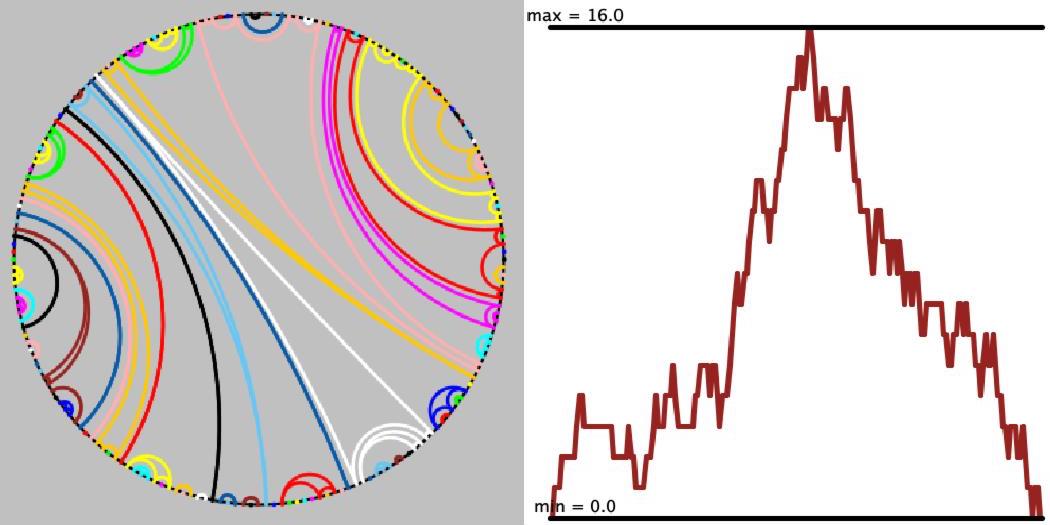}
              \caption{A random NC partition of $[200]$ and the corresponding width process.}
              \label{figWidthProcess}
\end{figure}

Remark: The results about the height of a random planar tree/Dyck path have been obtained in \cite{fgor93} by a difficult analysis of generating functions of trees that have limited height. An alternative approach to prove these results is to note that a random Dyck path converges uniformly almost surely to a Brownian excursion as $n \to \infty$, (see \cite{marchal2003}), and then use known results about Brownian excursions, as described in \cite{pitman_yor01}. 

Figure \ref{figWidthProcess} illustrates that the width of non-crossing partitions at $x$ converges as a process in $x$ to a Brownian excursion process.

%%%%%%%%%%%%%%%%%%%%%%%%%%%
%Appendix 
%
%%%%%%%%%%%%%%%%%%%%%%%%%%

\appendix

\section{A bijection between Dyck paths with $2n$ steps and binary planar trees with $2n + 1$ vertex}

%\subsection{Dyck paths with $2n$ steps and rooted planar trees with $n + 1$ vertices}
%
% It is easier to describe the inverse bijection, from trees to Dyck paths. First, every ordered rooted tree corresponds to a walk along the tree that visits every edge twice. We go on this walk starting from the root. We always attempt to go in such a way that the distance from the root is increasing and if it is possible then choose the left-most branch that we not yet been explored. If it is not possible to go with the increase in the distance, then the walker goes back to the root until an un-explored branch presents itself.  
%
%Given this walk, we construct a Dyck path by putting an ``up'' step if the walker moves away from the root and ``down'' step if she moves closer to the root. This provides a map from trees to Dyck paths which turns out to be a bijection. 

%\subsection{Dyck paths with $2n$ steps and rooted binary planar trees with $2n + 1$ vertices}

The map from paths to the set of these trees is defined recursively. Consider the point of the first return of the path to $0$. There are two possibilities: either it happens at the last step, or it happens before the last step. 

In the first case, the path has the form $U S D$, where $S$ is a Dyck path with $2 (n - 1)$ steps.  It follows that there is a binary tree $G(S)$ corresponding to $S$. We create a new binary tree by defining a new root vertex with left and right edges and connecting the tree $G(S)$ to the right edge. The left edge is marked by the label of the first ``up'' step in the Dyck path $U S D$. 

In  the second case, the Dyck path can be written as $S_1 S_2$ where $S_1$ and $S_2$ are two Dyck paths of lengths $2k_1, 2k_2 > 0$ with $k_1 + k_2 = n$, and the path $S_2$ has the property that it never returns to zero except at the last step. By recursion we can build binary trees $G(S_1)$ and $G(S_2)$ on $2 k_1 + 1$ and $2 k_2 + 1$ vertices. The tree $G(S_2)$ has an additional property that it has a leaf immediately on the left of its root. We create a new tree on $2 n + 1$ vertices by gluing the root of  $G(S_1)$ to this left leaf of $G(S_2)$. 

\bibliographystyle{plain}
\bibliography{comtest}

\end{document}